\documentclass[12pt,a4paper,final]{article}
\usepackage{amssymb, amsmath, amsthm}
\usepackage{mathrsfs}  
\usepackage{bbm}
\usepackage{graphicx}
\usepackage[ocgcolorlinks, linkcolor=red]{hyperref}
\usepackage{geometry}
 \usepackage{xcolor}
 \usepackage{tcolorbox}
\usepackage{mathrsfs}  

\usepackage{bm}
\usepackage{ifdraft}
\ifoptionfinal{
\usepackage[disable]{todonotes}
}{
\usepackage[norefs, nocites]{refcheck}

\usepackage[notref, notcite]{showkeys}
\usepackage[bordercolor=white, color=white]{todonotes}
}

\makeatletter\providecommand\@dotsep{5}\def\listtodoname{List of Todos}\def\listoftodos{\hypersetup{linkcolor=black}\@starttoc{tdo}\listtodoname\hypersetup{linkcolor=blue}}\makeatother
\usepackage{hyperref}
\usepackage{etoolbox}

\newtheorem{theorem}{Theorem}[section]
\newtheorem{lemma}[theorem]{Lemma}

\theoremstyle{remark}
\newtheorem{remark}{Remark}

\numberwithin{equation}{section}

\def\C{\mathbb C}
\def\R{\mathbb R}

\def\N{\mathbb N}

\renewcommand{\leq}{\leqslant}
\renewcommand{\geq}{\geqslant}

\def\p{\partial}

\newcommand*\xbar[1]{%
   \hbox{%
     \vbox{%
       \hrule height 0.5pt % The actual bar
       \kern0.5ex%         % Distance between bar and symbol
       \hbox{%
         \ensuremath{#1}%
       }%
     }%
   }%
} 

\date{}
\title{Recovery of Sturm-Liouville operators from partial boundary spectral data and applications}

\author{Ali Feizmohammadi \thanks{Department of Mathematics, University of Toronto, 3359 Mississauga Road, Mississauga, ON L5L1C6 (ali.feizmohammadi@utoronto.ca).} \and Yavar Kian \thanks{Univ Rouen Normandie, CNRS, Normandie Univ, LMRS UMR 6085, F-76000 Rouen, France. (yavar.kian@univ-rouen.fr)}}

\begin{document}

%\begin{titlepage}
%\maketitle
%\end{titlepage}

\maketitle

\begin{abstract}
	We study the inverse Sturm--Liouville problem on a finite interval from partial knowledge of spectral data. Specifically, we show that the potential can be uniquely reconstructed from the knowledge of a fraction of Dirichlet eigenvalues together with the normal derivatives of the corresponding eigenfunctions at both endpoints. We present two novel applications of our spectral result in  inverse coefficient determination problems for evolutionary PDEs that include passive wave-based imaging of a medium and active imaging for the time-dependent Schr\"odinger equation with unknown internal sources. Our results yield finite time measurement bounds for such inverse coefficient determination problems. A central innovation is the use of Kahane’s interpolation theorem to analyze endpoint time traces of solutions, enabling the recovery without requiring analyticity assumptions or infinite-time data, as in previous approaches. Finally, in the appendix, we present a spectral interpolation theorem for one-dimensional Schr\"odinger operators, which may be of independent interest.
\end{abstract}

%\tableofcontents
\section{Introduction and main results}

The inverse Sturm-Liouville problem, a classical yet perpetually rich area of mathematical analysis, centers on the reconstruction of the potential function \( V(x) \) and boundary conditions from spectral data associated with the Schr\"odinger operator $-\frac{d^2}{dx^2}+V(x)$ on $(0,1)$ subject to Robin boundary conditions at both end points. The spectral data typically consists of the eigenvalues $\{\lambda_n \}_{n=1}^\infty$ and corresponding norming constants$\{\gamma_n \}_{n=1}^\infty$, or alternatively, the Weyl--Titchmarsh \( m \)-function, see \cite{Mar52}. The theory started with a rigidity result proven by Ambarzumian in 1929 \cite{Ambarzumian1929berEF}; if the Neumann eigenvalues of a Schr\"odinger operator on the interval $(0,1)$ coincide with the sequence of numbers $n^2\pi^2$, $n=0,1,2,\ldots$, then the potential must be identically zero. However, this result turns out to be rather special. Indeed, the pioneering work of Borg~\cite{Borg1946EineUD} establishes that in the absence of symmetries for $V$, a single spectrum by itself does not suffice to uniquely determine \( V \); however, the knowledge of two spectra corresponding to different boundary conditions ensures uniqueness. Levinson~\cite{Levinson} refined this result by identifying conditions under which a single spectrum, augmented with boundary data extracted from the eigenfunctions, guarantees uniqueness. The landmark work of Gelfand and Levitan~\cite{GelLev51} introduced an integral equation method that not only proves uniqueness but also provides an explicit reconstruction algorithm from the full set of eigenvalues and norming constants. Another classical result which is also pertinent to our study is due to Hochstadt and Lieberman \cite{HL78}, which says that if the potential is known on half of the interval, one spectrum recovers the whole potential. We refer also to the work of Marchenko \cite{Mar52} for recovery of $V$ from its associated spectral measure or the Weyl function, and \cite{MCLAUGHLIN1988354} for recovery of the potential from the knowledge of nodal points of the spectrum. As we will discuss later, the inverse spectral problem above has profound connections to inverse coefficient determination problems for evolution equations. Let us make the terminological comment that throughout the paper we work with one-dimensional Schr\"odinger operators, which serve as the standard normal form of scalar Sturm-Liouville problems via the classical Liouville transformation.

A key achievement in the inverse spectral theory is the characterization theorem: given a sequence \( \{\lambda_n, \gamma_n\}_{n=1}^\infty \) satisfying certain asymptotic and positivity conditions, there exists a unique potential \( V \in L^2((0,1)) \) and boundary conditions such that $\lambda_n$'s correspond to the spectrum of the Sturm-Liouville problem and, without being explicit, the $\gamma_n$'s are related to some norming constants for its eigenfunctions. This result can be interpreted in terms of the inverse spectral map:
\begin{equation}\label{ip_spectral_formulation}
V \mapsto \left\{ \lambda_n, \gamma_n \right\}_{n=1}^\infty,
\end{equation}
which is proven to be bijective under appropriate regularity and normalization assumptions. This has a simple but profound implication; even the absence of one eigenvalue from the inverse spectral dataset leads to nonuniqueness. This initiated a new direction of research by introducing mixed-data where some fraction of the spectrum may be missing from our dataset but this is augmented by knowing some additional local information about the potential function $V$ in lieu of the missing spectral data.

A classical contribution in this area is due to Gesztesy and Simon~\cite{Gesztesy1999InverseSA}, who demonstrated that if \( V \) is known on slightly more than half of the interval, say $[1/2 - \varepsilon, 1]$ for some $\varepsilon \in (0,\frac{1}{2})$, then the knowledge of slightly more than $(1-2\varepsilon)$-fraction of the Dirichlet spectrum determines \( V \) uniquely on the entire interval. This result exemplifies a common feature of inverse problems with partial spectral data: 
\begin{itemize}
\item[{\bf(H)}]{\em Uniqueness of $V$ from partial knowledge of spectral data often follows when the potential is known on slightly more than half the domain.}
\end{itemize}
A seminal result due to Horváth \cite{Horvath2005InverseSP} (see also \cite{RGS97,Hald1984DiscontinuousIE,Horvath2001OnTI}) establishes that the uniqueness in the inverse spectral recovery of a potential may be characterized in terms of the completeness of a system of exponential functions, constructed from the given spectral data and partial knowledge of the potential, forging a deep connection between inverse spectral problems and the theory of completeness for exponential systems.  Let us also mention that the local Borg--Marchenko type results, developed by Gesztesy and Simon \cite{GS2000} and subsequently simplified by Christer Bennewitz \cite{Bennewitz2001APO}, establish uniqueness of a potential on a subinterval from exponential high-energy control of the Weyl--Titchmarsh function. More precisely, equality of the potential on $[0,a]$ with $a<1$ is equivalent to exponential decay of the difference of the corresponding Weyl functions as the spectral parameter tends to infinity along a non-real ray. These results encode agreement of full Cauchy data at an interior point in an analytic form.

In more recent work, Marletta and Weikard~\cite{Marletta_2005} investigated the weak stability of inverse Sturm--Liouville problems under finite and possibly noisy spectral data, extending the theory to include complex-valued potentials. Their results provide quantitative bounds on the approximation of a potential given finitely many eigenvalues with measurement errors. %Hatinouglu~\cite{Hatinouglu2019MixedDI} introduced the use of \emph{mixed} spectral data—eigenvalues corresponding to different boundary conditions—as a means to improve identifiability and uniqueness, further broadening the spectrum of admissible data in inverse problems.

\subsection{Inverse spectral result with partial data}
\label{sec_inv_spectral_intro}

Recalling the limitation {\bf{(H)}} for all partial data inverse spectral results above related to injectivity of \eqref{ip_spectral_formulation} (with $n$ belonging to a proper subset $S\subset \N$), we aim to obtain rather complete results by taking a particular norming constant in our dataset, $\gamma_n$, with $n\in S$ together with some local knowledge of the potential near an end point with the fundamental goal of removing the necessity of knowing the potential on more than half of the interval as in \cite{Gesztesy1999InverseSA}. 

To formulate our result, let $V \in L^\infty((0,1))$ be a real-valued potential, and let $\{\phi_k\}_{k=1}^\infty \subset W^{2,\infty}((0,1))$ be an $L^2((0,1))$-Schauder basis consisting of Dirichlet eigenfunctions of the one-dimensional Schr\"odinger operator defined uniquely via 
\begin{equation}\label{eigen_functions}
\left(-\frac{d^2}{dx^2} + V(x)\right)\phi_k(x)= \lambda_k\, \phi_k(x) \qquad \text{on } (0,1),
\end{equation}
subject to the boundary conditions
\begin{equation}\label{initial_data}
\phi_k(0)=\phi_k(1)=0, \qquad \text{and}\qquad  \p_x\phi_k(0)=1.\end{equation}
In this paper, we ask whether the knowledge of a fraction of the Dirichlet spectrum along with the knowledge of the normal derivative of the corresponding eigenfunctions on small intervals allow us to recover $V$ globally. Our main contribution is to establish uniqueness without requiring a priori knowledge of the potential on a large subinterval.

\begin{theorem}
	\label{thm_spectral}
	For $j=1,2,$ let $V_j\in L^\infty((0,1))$ be real-valued and let us denote by $\lambda_1^{(j)}<\lambda_2^{(j)}<\ldots$ the Dirichlet eigenvalues for $-\frac{d^2}{dx^2}  + V_j(x)$ on $(0,1)$. Let $\{\phi_k^{(j)}\}_{k=1}^\infty\subset W^{2,\infty}((0,1))$ be as in \eqref{eigen_functions}--\eqref{initial_data} with $V=V_j$. Let $S\subset \N$ satisfy 
	\begin{equation}\label{S_prop}\limsup\limits_{r\to \infty} \frac{|S\cap (0,r)|}{r}\geq 1-\varepsilon\qquad \text{for some $\varepsilon \in (0,1)$}.\end{equation} 
	Assume also that $\mathrm{supp}\,(V_1-V_2)\subset [0,1-\varepsilon)$. If,
	\begin{equation}\label{spec_hypo}
		\lambda_k^{(1)}= \lambda_{k}^{(2)} \quad \text{and}\quad \p_x\phi_{k}^{(1)}(1) =\p_x\phi_{k}^{(2)}(1) \qquad \forall\, k\in S,
	\end{equation}
then $V_1=V_2$ on $(0,1)$.
\end{theorem}

\begin{remark}
We emphasize that incorporating the normal derivatives of the eigenfunctions at both endpoints enables us to obtain a stronger result than that of \cite{Gesztesy1999InverseSA}. In particular, instead of requiring knowledge of the potential on at least half of the interval, it now suffices to know its values only within an $\varepsilon$-neighborhood of an endpoint.
\end{remark}

\begin{remark} Beyond its intrinsic spectral interest, Theorem~\ref{thm_spectral} yields powerful consequences in time-dependent inverse problems, particularly in imaging a medium with the additional presence of unknown sources and initial data. We next demonstrate its utility in determining both the initial condition and potential in wave-based inverse problems, a setting closely related to photoacoustic tomography, followed by a further application in the study of inverse problems for dynamical Schr\"odinger equations in the presence of unknown sources in the system.
\end{remark}
\subsection{Passive imaging of unknown media}
Inverse problems for time-dependent partial differential equations (PDEs) are concerned with the recovery of hidden properties of a system, such as coefficients, sources, or initial conditions, from limited observations of its evolution. These problems are foundational in applied mathematics, with wide-ranging applications from medical imaging to geophysics and materials science. Among the most challenging and practically relevant formulations are \emph{passive} inverse problems, where the observer has no control over the input signal and must infer internal characteristics of the medium solely from measurements taken at the boundary or within a localized observation region. A well known example of this type of inverse problem arises in the so-called photoacoustic tomography (PAT), a hybrid imaging modality that couples the high contrast of optical absorption with the resolution capabilities of ultrasound. The mathematical formulation of PAT is fundamentally an inverse problem for the wave equation, in which one seeks to reconstruct an unknown initial pressure distribution as well as an unknown wave speed or potential function from measurements of the resulting acoustic wave field on the boundary of the domain. Despite its origin in biomedical imaging, PAT has emerged as a rich source of inverse problems of both theoretical and applied interest.

In this work, we consider a one-dimensional model of PAT on a finite interval, where the acoustic wave propagation is governed by the standard wave equation with an initial pressure $f$ and with the potential $V$ representing the internal properties of the medium. We impose homogeneous Dirichlet boundary conditions, corresponding physically to an acoustically soft medium,
\begin{equation}\label{pf1}
	\begin{aligned}
		\begin{cases}
			\p_t^2 u-\p^2_xu+V(x)u=0 &\text{in } (0,T)\times(0,1),\\
			u(t,0)=u(t,1)=0 &\text{on } (0,T),\\
			u(0,x)=f(x), \quad \p_t u(0,x)=0 &\text{on } (0,1),
		\end{cases}
	\end{aligned}
\end{equation}
Given each $f\in H^1_0((0,1))$ the above problem admits a unique solution $u$ in the energy space
$$  C^1([0,T];L^2((0,1)) \cap C^0([0,T];H^1_0((0,1))).$$
The goal of the inverse problem is to simultaneously reconstruct unknown initial pressure $f$ and the potential $V$ from the boundary measurement 
$$  \p_x u(t,0) \quad \text{and}\quad \p_x u(t,1) \quad t\in (0,T).$$
We obtain the following two uniqueness results. Our first result only requires a condition on the initial data at the end points of the interval as well as some additional regularity on the potential function. The result is stated in a rigidity sense, meaning that we only impose the conditions on one pair of data and not the other. 
\begin{theorem}
	\label{thm_1}
	Let $T>2$. For $j=1,2$, let $V_j\in L^\infty((0,1))$ be real-valued, and let $f_j\in H^1_0((0,1)))$. Assume that $V_1\in C^4([0,1])$, $f_1\in H^2_0((0,1))\cap H^3((0,1))$ and that $|f_1''(0)|\neq |f_1''(1)|$ \footnote{Following the proof, this assumption can be changed to the alternative condition that $|f_1^{(2N)}(0)|\neq |f_1^{(2N)}(1)|$ for some $N\in \N$, provided that $V_1\in C^{2N+2}([0,1])$, $f_1\in H^{2N}_0((0,1))\cap H^{2N+1}((0,1))$.}. Assume also that $V_1=V_2$ in a neighborhood of  $x=1$. Let $u_j$ denote the solution of \eqref{pf1} with $f=f_j$ and $V=V_j$. If
	\begin{equation}\label{DN_eq_1}
		\p_x u_1(t,0)=\p_x u_2(t,0) \quad \text{and} \quad \p_x u_1(t,1)=\p_x u_2(t,1) \quad \forall\,t\in(0,T),
	\end{equation}
	then $f_1=f_2$ and $V_1=V_2$ on $(0,1)$.
\end{theorem}

Our second result addresses the inverse problem under a support condition on one of the initial data, requiring it to be localized near one endpoint. This assumption is physically motivated; for instance, in seismic imaging, earthquakes typically originate near the Earth's surface.
\begin{theorem}
	\label{thm_1_support}
	Let $T>2$. For $j=1,2$, let $V_j\in L^\infty((0,1))$ be real-valued, and let $f_j\in H^1_0((0,1)))$. Assume that $f_1$ is not identical to zero on $(0,1)$ and that
	\begin{equation}\label{f_1_V_1_cond}
		\mathrm{supp}\, f_1 \subset [0,\varepsilon] \quad \text{and}\quad \mathrm{supp}\,(V_1-V_2) \subset [0,1-\varepsilon),
	\end{equation}
	for some $\varepsilon\in (0,1)$. Let $u_j$ denote the solution of \eqref{pf1} with $f=f_j$ and $V=V_j$. If
	\begin{equation}\label{DN_eq_supp}
		\p_x u_1(t,0)=\p_x u_2(t,0) \quad \text{and} \quad \p_x u_1(t,1)=\p_x u_2(t,1) \quad \forall\,t\in(0,T),
	\end{equation}
	then $f_1=f_2$ and $V_1=V_2$ on $(0,1)$.
\end{theorem}

The existing literature on passive inverse problems remains comparatively limited. Classical contributions by Pierce~\cite{Pierce1979UniqueIO}, Suzuki~\cite{Suzuki1983UniquenessAN, Suzuki1986InversePF}, and Murayama~\cite{Murayama1981TheGT} address parabolic models under strong spectral assumptions on the initial data $f$, with analogous considerations in the Schrödinger setting by Avdonin et al.~\cite{Avdonin2011ReconstructingTP, Avdonin2010}. In all these works, the spectral condition imposed on the initial data  $f_j$, $j=1,2$, requires that
$$ \left(f_j,\phi_k^{(j)}\right)_{L^2((0,1))} \neq 0 \quad \forall\,k\in \N.$$
which constitutes a non-degeneracy assumption. This condition is both necessary and sufficient for uniqueness in the one-dimensional problem. More recently, the authors of~\cite{FeiKrup25} established analogous uniqueness results for multidimensional inverse problems with passive measurements. Their contribution represents a genuine breakthrough: whereas the one-dimensional case rests essentially on tools from complex analysis, the multidimensional theory is developed through entirely different techniques that open a new methodological direction.

Returning to the nondegeneracy condition above in the one dimensional setup, we remark that such hypotheses, while mathematically generic, are generally unverifiable and physically opaque. In contrast, our result on the above one dimensional inverse problem only relies on concrete verifiable physical conditions on the initial data, avoiding reliance on inaccessible spectral information. Our approach again solely relies on methods in complex analysis which signals that such results may not be possible to obtain in multidimensional setups.

The work closest in spirit to ours is the recent contribution of the first author~\cite{Feiz25}, which studies an analogous passive inverse problem in the context of 1-D parabolic equations but with measurement of the normal derivative of the solution at only one end point. The author shows that under an additional assumption on the size of the support of the unknown initial data, the problem may be reduced to the question of reconstructing a potential from a fraction of the Dirichlet spectrum. That reduction relies on a Paley--Wiener correspondence between the spatial support of initial data and vanishing conditions on spectral components. However, uniqueness in~\cite{Feiz25} hinges on knowing the potential on more than half the domain, a limitation arising from the exclusive use of eigenvalue data. The seminal result of Gesztesy and Simon~\cite{Gesztesy1999InverseSA} illustrates the necessity of this constraint in the absence of eigenfunction information. 

A key difference between the present work and \cite{Feiz25} is the incorporation of boundary data at both endpoints, which effectively encodes partial information on the eigenfunctions. This enables us to apply our refined inverse spectral result (Theorem~\ref{thm_spectral}) to achieve uniqueness without prior knowledge of $V$ on a large portion of the domain. Another novel aspect of our result lies in the finite time measurement intervals appearing in our theorems. This improvement is made possible by introducing a new technique for analyzing the solution at the endpoints in the form of a time series. Our approach leverages a Paley--Wiener interpolation theorem due to Kahane~\cite{Kahane1957SurLF, Kahane1962}, which enables us to bypass the need for infinite-time data or analyticity assumptions previously required in~\cite{Feiz25}. Let us also point out that the condition $T>2=2\,\textrm{Diam}((0,1))$ is a classical assumption for inverse coefficient determination problems for the wave equation and it is connected with the finite speed of propagation, see e.g. \cite{KMO}.

We close this section by mentioning that our results also resonate with contemporary advances in passive imaging using stochastic sources. Helin et al.~\cite{Helin2013InversePF, Helin2016CorrelationBP}, and more recently Bl\aa sten et al.~\cite{Bla25}, demonstrate that certain random wavefields carry sufficient information for inverse recovery. However, our analysis is fully deterministic, and we establish that even a single, uncontrolled excitation, when properly interpreted through spectral methods, suffices to uniquely determine both the medium and the initial data.

\subsection{Active imaging in the presence of unknown sources}
Inverse problems involving active measurements concern the determination of internal features of an unknown medium by applying a prescribed perturbation, such as an initial displacement or boundary excitation, and subsequently recording the system’s response, often at the boundary or within a restricted subset of the domain. When the input signal is sufficiently rich or varied, one may expect the associated input-output data to encode detailed information about underlying properties, such as variable coefficients or embedded structures.

In many practical settings, however, precise control over the system is limited or altogether lacking. This may be due to inherent background dynamics, environmental interference, or technical constraints that prevent the full specification or isolation of the initial state. For example, in seismological applications, while it is possible to introduce a controlled impulse at the surface of the Earth, the resulting measurements are typically superimposed with ambient and persistent fluctuations, both natural and anthropogenic, which are neither localized in space nor confined in time. Consequently, the data obtained from such an experiment reflect the combined effect of the intended perturbation and an ensemble of unknown, temporally extended, and spatially distributed influences. This situation motivates a fundamental question:

\vspace{1ex}
\begin{itemize}
	\item[\bf(Q3)] \emph{Can one determine the internal structure of a medium, such as its potential, wave speed, or geometry, from a single active boundary measurement, even in the presence of unknown and possibly persistent internal sources?}
\end{itemize}
\vspace{1ex}

The problem stands at the intersection of inverse scattering theory, microlocal analysis, and the spectral theory of partial differential equations. It challenges the classical paradigm in which identifiability is predicated on access to a full set of controllable inputs. Instead, it calls for techniques capable of extracting invariant information from data contaminated by unknown and uncontrollable contributions.

In this paper, we introduce a rigorous framework to address this question in the setting of one-dimensional Schr\"odinger equations. Specifically, we consider the recovery of an internal potential on the interval $(0,1)$, based on a single boundary observation of the solution, under the assumption that the system evolves from arbitrary, unknown initial data. To model the active measurement, we introduce a transient, spacetime localized perturbation near the boundary point $x=0$ and near $t=0$. This is encoded by a term of the form $\mathbbm{1}_\delta(t)\,\mathbbm{1}_\delta(x)$ with $\delta\in (0,1)$, where
\begin{equation}\label{1_epsilon}
	\mathbbm{1}_\delta(t)=\begin{cases}
		1 &\text{if } t\in[0,\delta],\\
		0 &\text{otherwise}.
	\end{cases}
\end{equation}
This construction serves as a mathematically idealized model of a localized probing signal, capable of eliciting a detectable response from the medium without suppressing or isolating the unknown components of the initial state. To state our result, we consider the initial boundary value problem
\begin{equation}\label{pf2}
	\begin{aligned}
		\begin{cases}
			\mathrm{i}\,\p_t u-\p_x^2 u+V(x)u=F(x)+\mathbbm{1}_{\delta}(t)\,\mathbbm{1}_{\delta}(x) &\text{on } (0,T)\times(0,1),\\
			u(t,0)=u(t,1)=0 &\text{on } (0,T),\\
			u(0,x)=f(x) &\text{on } (0,1).
		\end{cases}
	\end{aligned}
\end{equation}
Here, the initial data $f$, the source $F$ and the potential $V\in L^\infty((0,1))$ are all assumed to be a priori unknown. We remark that the source $F$ may be physically interpreted as modelling some consistent time invariant background noise in the system. Given each $f\in H^1_0((0,1))$ and $F\in L^2((0,1))$, the above problem admits a unique solution in the energy space
$$
C^0([0,T];H^1_0((0,1))).
$$
Moreover, $\p_x u|_{x=0,1} \in L^2((0,T))$. The inverse problem is concerned with reconstructing the potential $V$ as well as the unknown initial data $f$ and the source $F$ given the single short range boundary measurement 
$$  \p_x u(t,0) \quad \text{and}\quad \p_x u(t,1) \quad t\in (0,\delta),$$
where $0<\delta<T$. We prove the following theorem.
\begin{theorem}
	\label{thm_2}
	Let $T>0$. For $j=1,2$, let $V_j\in L^\infty((0,1))$ be real-valued, let $f_j\in H^1_0((0,1))$ and let $F_j \in L^2((0,1))$. Assume that $V_1\in C^4([0,1])$, $f_1\in H^3_0((0,1))$ and that $F_1\in H^1_0((0,1))$. Assume also that $V_1=V_2$ in a neighbourhood of $x=1$. Let $\delta\in (0,1)$ be sufficiently small \footnote{Assuming that $\mathrm{supp}\,(V_1-V_2)\subset [0,1-\varepsilon)$ for some $\varepsilon>0$, it suffices to take $0<\delta<\min\{1,2\varepsilon\}$.}. Let $u_j$ denote the solution to \eqref{pf2} with $V=V_j$, $f=f_j$ and $F=F_j$. If
	\begin{equation}\label{DN_eq_2}
		\p_x u_1(t,0)=\p_x u_2(t,0) \quad \text{and} \quad \p_x u_1(t,1)=\p_x u_2(t,1) \quad \forall\,t\in(0,\delta),
	\end{equation}
	then $V_1=V_2$ on $(0,1)$\footnote{Our proof also shows that the initial data and the source could be recovered up to the natural gauge for the problem, namely that $F_1-F_2=(-\p^2_x+V)(f_1-f_2)$, where $V:=V_1=V_2$.}.
\end{theorem}

Most of the previous results on this problem are in multidimensional settings and assume that no other unknown sources or unknown phenomena are present in the system and focus on using a single active measurement to reconstruct the coefficients of the PDE. We mention for example the foundational work of Bukhgeim and Klibanov, who introduced Carleman-based techniques to establish global uniqueness from single measurements~\cite{BukKli81} under a positivity type assumption for the initial data. Subsequent developments have demonstrated that, under suitable geometric or structural assumptions, a single specially designed, often highly singular, measurement can suffice for uniqueness in a wide array of multidimensional settings. For hyperbolic equations, such results include the determination of a wave speed or time-dependent coefficient from one boundary measurement via Carleman estimates and microlocal analysis~\cite{Stefanov2011RecoveryOA,Bellassoued2008DeterminationOA,Feizmohammadi2020GlobalRO}. In parabolic and diffusion-type problems, single measurement identification has been achieved for convection terms or simultaneously for multiple coefficients~\cite{Cheng2002IdentificationOC,Kian2022SimultaneousDO}. Extensions to fractional evolution equations have revealed that even the order of differentiation can be uniquely recovered~\cite{Kian2020TheUO}. In 1-D or complex coefficient settings, refined techniques ensure uniqueness despite limited data~\cite{Rakesh2001AOI}. 

\begin{remark}
	In higher dimensions, particularly in the context of wave or Schr\"odinger equations, a common strategy involves employing a dense family of highly singular sources, which are then carefully combined into a single source. This construction allows for the extraction of detailed information as the resulting singularities propagate through the medium. The multidimensional propagation of singularities generates a rich dataset, reflecting the intricate structure of the underlying dynamics. We stress, however, that this phenomenon is fundamentally absent in one dimension: methods relying solely on the propagation of singularities provide, at best, severely limited information about the coefficients.
\end{remark}

\section{Proof of Theorem~\ref{thm_spectral}}
\label{sec_thm_spectral}
We will assume throughout this section that the hypotheses of Theorem~\ref{thm_spectral} is satisfied. For each $j=1,2,$ and each $z\in \C$, let us define $\psi^{(j)}(\cdot,z)$ as the unique solution of the equation
\begin{equation}\label{Psi_def_z}
		\left(-\frac{d^2}{dx^2} + V_j(x)\right)\psi^{(j)}(x,z)= z^2\, \psi^{(j)}(x,z) \qquad \text{for all $x\in (0,1)$ },
	\end{equation}
	subject to the boundary conditions
	\begin{equation}\label{psi_initial_data}
	\psi^{(j)}(0,z)=0 \qquad \text{and}\qquad  \p_x\psi^{(j)}(0,z)=1.\end{equation}
It is well known that $\psi$ depends analytically on $z\in \C$ and that there exists a constant $C>0$ depending on $\max\{\|V_1\|_{L^\infty((0,1))}, \|V_2\|_{L^\infty((0,1))}\}$ such that the following estimates hold for each $j=1,2,$ uniformly on $[0,1]\times \C$,
\begin{equation}
	\label{ptwise_estimate_1}
	\left|\psi^{(j)}(x,z)-\frac{\sin (zx)}{z}\right| \leq  \frac{C}{1+|z|^2}\,e^{|\mathrm{Im}\,z|\,x} \quad \forall\, x\in [0,1] \quad \forall\, z\in \C,
\end{equation}
and
\begin{equation}
	\label{ptwise_estimate_2}
	\left|\p_x\psi^{(j)}(x,z)-\cos(zx)\right| \leq  \frac{C}{1+|z|}\,e^{|\mathrm{Im}\,z|\,x}\quad \forall\, x\in [0,1] \quad \forall\, z\in \C.
\end{equation}
Let us also observe that by the hypotheses of the theorem, there exists $\varepsilon_1\in (\varepsilon,1)$ such that:
\begin{equation}
	\label{psi_hypo}
	\psi^{(1)}(x,\sqrt{\lambda_k^{(1)}}) = 	\psi^{(2)}(x,\sqrt{\lambda_k^{(1)}}) \quad \forall\, x\in [1-\varepsilon_1,1]\quad \text{and}\quad \forall\, k\in S,
\end{equation}
and also that 
\begin{equation}
	\label{psi_hypo_1}
	\p_x\psi^{(1)}(x,\sqrt{\lambda_k^{(1)}}) = \p_x	\psi^{(2)}(x,\sqrt{\lambda_k^{(1)}}) \quad \forall\, x\in [1-\varepsilon_1,1]\quad \text{and}\quad \forall\, k\in S,
\end{equation}
Let us now recall a deep result of Levinson \cite[Chapter II, Theorem VIII, page 13]{Levinson1940} regarding the distribution of zeros of entire functions under certain growth estimates (see also \cite[Chapter 4, page 173]{Levin64}). Before writing the statement of the theorem of Levinson, we define the function $\log^+:\R \to [0,\infty)$ via
$$
\log^+x = \begin{cases}
	\log x \quad \text{if $x>1$}\\
	0		\quad \text{otherwise}.
\end{cases}
$$
\begin{theorem}[Chapter II, Theorem VIII, \cite{Levinson1940}]
	\label{thm_levinson}
	Let $F(z)$ be an entire function that is not identical to zero. Assume that 
	\begin{equation}\label{logarith_growth}
		\int_\R 	\frac{\log^+|F(x)|}{1+x^2}\,dx<\infty,
	\end{equation}
	and that
	\begin{equation}\label{lim_sup_cond}
		\limsup_{r\to \infty}\frac{\log |F(re^{\mathrm{i}\theta})|}{r} \leq k.
	\end{equation}
	Let $n(r)$ be the number of zeros of the function $F(z)$ that lie in the region 
	$$\{z\in \C\,:\, \mathrm{Re}\,z\geq 0\quad |z|< r \}.$$ 
	Then, there exists a number $B\leq \frac{k}{\pi}$ such that, 
	$$
	\lim_{r\to \infty} \frac{n(r)}{r}=B \leq \frac{k}{\pi}.
	$$
\end{theorem}

We are now ready to prove our main inverse spectral result.

\begin{proof}[Proof of Theorem~\ref{thm_spectral}]
	Let us define the entire function $F:\C \to \C$ via
	\begin{equation}\label{F_entire_def}
	F(z) = \psi^{(1)}(1-\varepsilon_1,z)- \psi^{(2)}(1-\varepsilon_1,z) \qquad \forall\, z\in \C,
	\end{equation}
and recall from \eqref{psi_hypo} that there holds:
	\begin{equation}\label{F_hypo}
	F(\sqrt{\lambda_k^{(1)}})=0 \qquad \forall\, k\in S.
\end{equation}
We claim that $F=0$ on $\C$. We give a proof by contradiction and assume for contrary that it is not identical to zero. For each $r>0$, let us define 
	$$ n(r)= \left|\left\{z\in \C\,:\, \mathrm{Re}\, z\geq 0 \quad \text{and}\quad  F(z)=0 \quad \text{and}\quad |z|<r\right\}\right|.$$
	Recalling the eigenvalue asymptotic expression (see e.g. \cite[Theorem 4, page 35]{Pschel1986InverseST}),
	\begin{equation}\label{eigen_asymp}
		\sqrt{\lambda_k^{(j)}}= k\pi + O(\frac{1}{k}) \quad \text{as $k\to \infty$}, \quad j=1,2,
		\end{equation}
	together with \eqref{F_hypo} and \eqref{S_prop}, we deduce that
	\begin{equation}\label{limsup_n}
	\limsup_{r\to \infty} \frac{n(r)}{r}\geq \frac{1-\varepsilon}{\pi}.
	\end{equation}
On the other hand, we know from \eqref{ptwise_estimate_1} that $F$ is uniformly bounded on the real axis which implies that \eqref{logarith_growth} is satisfied for the function $F$ above. Moreover, by applying \eqref{ptwise_estimate_1} for $j=1,2,$ and using the triangle's inequality, there holds
	$$
	|F(z)| \leq \frac{2C}{1+|z|^2} e^{(1-\varepsilon_1)|z|} \qquad \forall\, z\in \C,
	$$
	implying that for each $\theta \in [0,2\pi]$ we have
	$$
	\limsup_{r\to \infty}\frac{\log |F(re^{\mathrm{i}\theta})|}{r} \leq 1-\varepsilon_1.
	$$
	Thus, by Theorem~\ref{thm_levinson}, we deduce that $\lim_{r\to \infty} \frac{n(r)}{r}$ exists and that
	$$
	\lim_{r\to \infty} \frac{n(r)}{r} \leq \frac{1-\varepsilon_1}{\pi},
	$$
	which is a contradiction to \eqref{limsup_n} as $\varepsilon_1 \in (\varepsilon,1)$. Thus, $F$ must be identically zero which implies that
		\begin{equation}\label{F_vanish}
	\psi^{(1)}(1-\varepsilon_1,z) = \psi^{(2)}(1-\varepsilon_1,z) \quad \forall\, z\in \C.
	\end{equation}
	We may now repeat the same arguments as above but this time for the function 
	$$G(z)= \p_x\psi^{(1)}(1-\varepsilon_1,z) -\p_x\psi^{(2)}(1-\varepsilon_1,z) \quad \forall\, z\in \C,$$
	to conclude that $G$ must also vanish identically. This implies that 
	\begin{equation}\label{G_vanish}
	\p_x\psi^{(1)}(1-\varepsilon_1,z) = \p_x\psi^{(2)}(1-\varepsilon_1,z) \quad \forall\, z\in \C.
	\end{equation}
	Combining \eqref{F_vanish}-\eqref{G_vanish}, we deduce that the Weyl functions 
	$$m_j(z)= 	\frac{\p_x\psi^{(j)}(1-\varepsilon_1,z)}{\psi^{(j)}(1-\varepsilon_1,z)} \quad j=1,2,$$ 
	for the two potentials are equal as meromorphic functions in the complex plane, which implies that $V_1=V_2$ on $(0,1)$, see e.g. \cite{Mar52}.
\end{proof}

\section{Passive imaging: proof of Theorems~\ref{thm_1}-\ref{thm_1_support}}
\label{sec_passive_pf}
We begin with a lemma. 
\begin{lemma}
\label{lem_completeness}
Let $T>2$. For $j=1,2,$ let $V_j\in L^\infty((0,1))$ and denote by $\lambda_1^{(j)}<\lambda_2^{(j)}<\ldots$, the Dirichlet eigenvalues of the operator $-\frac{d^2}{dx^2}+V_j(x)$ on $(0,1)$. There exists 
\begin{equation}\label{eta_support}
	\eta_m\in L^2(\mathbb R)\quad \text{with}\quad \mathrm{supp}\,\eta_m \subset [0,T], 
\end{equation}
such that
\begin{equation}\label{eta_eq_lem}
	\begin{aligned}
\int_0^{T}\, \eta_m(t) \cos(\sqrt{z}t)\,dt&=0 \quad \forall\, z\in \left(\{\lambda_k^{(1)}\}_{k=1}^\infty\cup  \{\lambda_k^{(2)}\}_{k=1}^\infty\right)\setminus \{\lambda_m^{(1)}\} \\
\int_0^{T}\, \eta_m(t) \cos(\sqrt{\lambda_m^{(1)}}t)\,dt&=1.
\end{aligned}
\end{equation}	
Here, we recall that a finite number of $\lambda_k^{(j)}$'s could also be negative and so the square root and cosine functions need to be understood in the general sense.
	\end{lemma}

\begin{proof}
	Let us define for each $j=1,2,$ the set 
	$$\Gamma_j=\{\pm \sqrt{\lambda_k^{(j)}}\,:\, \lambda_k^{(j)}>0\}.$$
	Recalling the eigenvalue asymptotic expression \eqref{eigen_asymp}, we deduce that $\Gamma_j$ is uniformly discrete (i.e. the pairwise distances of its elements has a uniform lower bound) and that
	$$ D^+(\Gamma_j) := \lim_{r\to\infty} \max_{x\in \R}\frac{|\Gamma_j\cap (x,x+r)|}{r}= \frac{1}{\pi}<\frac{T}{2\pi}.$$
	Applying the main result of \cite{Kahane1957SurLF} together with \cite[Theorem 7, pg. 129]{Young1980AnIT} it follows that for each fixed $j\in \{1,2\}$, the set of complex exponentials $\{e^{\pm\mathrm{i}\sqrt{\lambda_k^{(j)}}t}\}_{k\in \N}$ is not complete in $L^4((-\frac{T}{2},\frac{T}{2}))$. Combining this with \cite[Lemma 5.4]{Horvath2005InverseSP} we deduce that for each fixed $j\in \{1,2\}$, the set $\{ \cos(\sqrt{\lambda_k^{(j)}}t)\}_{k\in \N}$ is not complete in $L^4((0,\frac{T}{2}))$. It follows that there exists a nonzero function $\tilde{h}^{(j)}_m\in L^4((0,\frac{T}{2}))$ with $j=1,2$, such that  
	\begin{equation}
		\label{tilde_h_m_eq}
		\begin{aligned}
			\int_0^{\frac{T}{2}}\, \tilde{h}^{(1)}_m(t) \cos(\sqrt{z}t)\,dt&=0 \quad &\forall\, z\in \{\lambda_k^{(1)}\}_{k=1}^\infty\\
				\int_0^{\frac{T}{2}}\, \tilde{h}^{(2)}_m(t) \cos(\sqrt{z}t)\,dt&=0 \quad &\forall\, z\in \{\lambda_k^{(2)}\}_{k=1}^\infty
		\end{aligned}	
	\end{equation}
Here, we are also identifying $\tilde{h}^{(j)}_m$ as a function on all of $\mathbb R$ by setting it to be zero outside $(0,\frac{T}{2})$. Let us now define the even function $h_m^{(j)} \in L^4((-\frac{T}{2},\frac{T}{2}))$ for each $j=1,2,$ (also viewed as a function on $\mathbb R$ by setting it to be zero outside of $(-\frac{T}{2},\frac{T}{2})$) via
$$
h_m^{(j)}(t)= \begin{cases}
	\tilde{h}^{(j)}_m(t)\quad &\forall\, t\in (0,\frac{T}{2})\\
	\tilde{h}^{(j)}_m(-t) \quad &\forall\, t\in (-\frac{T}{2},0).
\end{cases}
$$
We can rewrite \eqref{tilde_h_m_eq} as follows,
	\begin{equation}
	\label{h_m_eq}
	\begin{aligned}
		\int_{-\frac{T}{2}}^{\frac{T}{2}}\, h^{(1)}_m(t) e^{\pm\mathrm{i}\sqrt{z}t}\,dt&=0 \quad &\forall\, z\in \{\lambda_k^{(1)}\}_{k=1}^\infty\\
			\int_{-\frac{T}{2}}^{\frac{T}{2}}\, h^{(2)}_m(t) e^{\pm\mathrm{i}\sqrt{z}t}\,dt&=0 \quad &\forall\, z\in \{\lambda_k^{(2)}\}_{k=1}^\infty
\end{aligned}
\end{equation}
Let us now define $\tilde{\eta}_m\in L^2(\mathbb R)$ with $\textrm{supp}\, \tilde{\eta}_m \subset (-T,T)$ by the expression 
\begin{equation}\label{tilde_eta_m}
	\tilde{\eta}_m(t) = \int_{-T}^T h^{(1)}_m(t-\tau)\,h^{(2)}_m(\tau)\,d\tau \quad \forall\,t\in \mathbb R.
	\end{equation}
Next, we define the entire function 
\begin{equation}\label{F_def_entire}
	F_m(z) = \int_{-T}^{T} \tilde{\eta}_m(t) e^{-\mathrm{i}zt}\,dt \quad \forall\, z\in \C.	
	\end{equation}
Combining \eqref{h_m_eq} and the convolution definition \eqref{tilde_eta_m} it follows that 
$$
F_m(z)=0 \quad \forall\, z\in \left\{\pm\sqrt{\lambda^{(1)}_k}\right\}_{k\in \N} \cup \left\{\pm\sqrt{\lambda^{(2)}_k}\right\}_{k\in \N}.
$$
Suppose that $z=\sqrt{\lambda_m^{(1)}}$ is a zero of order $s$ for the function $F(z)$. If $F(\sqrt{\lambda^{(1)}_m})$ is not zero, then we simply define $s=0$. Note also that $F(z)$ is an even function of $z$ since $\tilde{\eta}_m$ is an even function and the Fourier transform of an even function is also even.  With this in mind, let us define the entire function
$$
G_m(z) = \frac{F_m(z)}{(z^2-\lambda_m^{(1)})^s}  \quad \text{if $\lambda_m^{(1)} \neq 0$},
$$
and alternatively by 
$$
G_m(z) = \frac{F_m(z)}{z^s}  \quad \text{if $\lambda_m^{(1)}=0$}.
$$
We remark that in the latter case $s$ will be either zero or an even number. Thus, $G_m(z)$ is an even function in both cases. We also record that 
$$
G_m(\sqrt{\lambda_m^{(1)}}) \neq 0.
$$
As the function $$H_m(z)=\frac{2G_m(z)}{G_m(\sqrt{\lambda_m^{(1)}})},$$ is of exponential type $T$ and it is square-integrable along horizontal lines in the complex plane, it follows from the Paley-Wiener theorem that $H_m$ is the Fourier transform of some $\hat{\eta}_m \in L^2((-T,T))$. Thus,
	\begin{equation}
	\label{hat_eta_m_eq}
	\begin{aligned}
		&\int_{-T}^{T}\, \hat{\eta}_m(t)\,e^{\pm\mathrm{i}\sqrt{z}t}\,dt=0 \quad \forall\, z\in \left(\{\lambda_k^{(1)}\}_{k=1}^\infty\cup  \{\lambda_k^{(2)}\}_{k=1}^\infty\right)\setminus \{\lambda_m^{(1)}\} \\
		&\int_{-T}^{T}\, \hat{\eta}_m(t) e^{\pm \mathrm{i}\sqrt{\lambda^{(1)}_m}t}\,dt=2.
	\end{aligned}
\end{equation}
Defining $\eta_m \in L^2((0,T))$ to be the restriction of $\hat{\eta}_m$ to the interval $(0,T)$ the claim follows, since $\hat{\eta}_m$ is an even function (inverse Fourier transform of an even function is even). 
\end{proof}

We also need the following lemma.

\begin{lemma}
	\label{lem_zero_density_wave}
	Let $V \in C^4([0,1])$ and let us denote by $\lambda_1<\lambda_2<\ldots$ the Dirichlet eigenvalues for the operator $-\frac{d^2}{dx^2}+V(x)$ on $(0,1)$. Let $\phi_k\in C^6([0,1])$ be as defined by \eqref{eigen_functions}-\eqref{initial_data}. Let $f\in H^2_0((0,1))\cap H^3((0,1))$ and assume that $|f''(0)| \neq |f''(1)|$. Then, there exists $N\in \N$, such that 
	\begin{equation}\label{eq_nonzero_modes}
		\int_0^1 f(x)\,\phi_k(x) \,dx \neq 0 \quad \forall\, k\geq N.
		\end{equation}
\end{lemma}

\begin{proof}
	It is well known (see e.g. \cite[Section 3, Section 4]{FULTON1994297})that if $V\in C^4([0,1])$, then there exists $\{A_\ell\}_{\ell=1}^3, \{B_\ell\}_{\ell=1}^3  \subset C^3([0,1])$ depending explicitly on $V$ such that given any $x\in [0,1]$ and $k\in \N$ there holds
	\begin{equation*}\label{eigen_expansion_lem}
		\phi_k(x)= \frac{\sin(k\pi x)}{k \pi}+\sum_{\ell=1}^3\left(\frac{A_\ell(x)}{k^{\ell+1}}\sin(k\pi x)+ \frac{B_\ell(x)}{k^{\ell+1}}\cos(k\pi x)\right) +O(k^{-5}),
	\end{equation*}
	where the modulus of convergence above (as $k\to \infty$) is uniform with respect to $x\in [0,1]$. Let us define for each $k\in \N$,
	$$a_k = \int_0^1 f(x)\,\frac{\sin(k\pi x)}{k\pi} \,dx,$$
	and 
	$$b_k = \int_0^1 f(x)\,\left(\sum_{\ell=1}^3\frac{A_\ell(x)}{k^{\ell+1}}\sin(k\pi x)+ \frac{B_\ell(x)}{k^{\ell+1}}\cos(k\pi x)\right)\,dx.$$
	Applying integration by parts, it follows that
	\begin{equation}
	\begin{aligned}
	a_k&= \frac{1}{k^2\pi^2 }\,\int_0^1 f'(x)\,\cos(k\pi x)\,dx \\
	&= -\frac{1}{k^3\pi^3}\,\int_0^1 f''(x)\sin(k\pi x)\,dx\\
	&=\frac{1}{k^4\pi^4}\left(f''(1)\cos(k\pi)-f''(0) \right)-\frac{1}{k^4\pi^4}\,\underbrace{\int_0^1 f'''(x)\,\cos(k\pi x)\,dx}_{\text{$o(1)$ as $k\to \infty$}}.
\end{aligned}	
\end{equation}
Note that 
$$
f''(1)\cos(k\pi)-f''(0)= f''(1)-f''(0) \neq 0 \quad \text{if $k$ is even},
$$
and 
$$
f''(1)\cos(k\pi)-f''(0)= -f''(1)-f''(0) \neq 0 \quad \text{if $k$ is odd}.
$$	
Using integration by parts again, it is straightforward to see that
$$
|b_k| =O(\frac{1}{k^5}) \quad \text{as $k\to \infty$}.
$$
The claim follows immediately from combining the above observations.
\end{proof}

We are ready to prove the theorem.

\begin{proof}[Proof of Theorem~\ref{thm_1}]
	We start by recalling for each $j=1,2,$ that
	\begin{equation}\label{u_12_exp}
				u_j(t,x)=\sum_{k=1}^\infty a_k^{(j)} \cos(\sqrt{\lambda_k^{(j)}}\,t)\phi_k^{(j)}(x), \quad t\in (0,T)\quad x\in (0,1),
		\end{equation}
	where 
		\begin{equation}\label{a_12_exp}
	a_k^{(j)} = \frac{1}{\|\phi_k^{(j)}\|^2_{L^2((0,1))}}\left(f_j, \phi_k^{(j)} \right)_{L^2((0,1))} \quad \forall\, k\in \N,
\end{equation}
and the convergence in the above infinite series is to be understood with respect to the  
$$C^1([0,T];L^2((0,1)))\cap C^0([0,T];H^1_0((0,1)))$$
topology. Recall from Lemma~\ref{lem_zero_density_wave} (with $V=V_1$ and $f=f_1$) that there exists $N \in \N$ such that
\begin{equation}\label{a_k_1}
a_k^{(1)} \neq 0 \quad \forall\, k \geq N.
\end{equation}
We claim that 
\begin{equation}\label{claim_eigen_wave}
	\lambda_k^{(1)} \in \{\lambda_\ell^{(2)}\}_{\ell=1}^\infty \quad \forall\, k\geq N. 	
	\end{equation}
We give a proof by contradiction. Suppose for contrary that there exists $m\geq N$ such that $\lambda_m^{(1)} \notin \{\lambda_\ell^{(2)}\}_{\ell=1}^\infty$. By Lemma~\ref{lem_completeness} there exists $\eta_m \in L^2((0,T))$ such that \eqref{eta_eq_lem} holds. We deduce via \eqref{DN_eq_1} that 
$$
\left(\p_x u_1(\cdot,0),\overline{\eta_m(\cdot)}\right)_{L^2((0,T))} = \left(\p_x u_2(\cdot,0),\overline{\eta_m(\cdot)}\right)_{L^2((0,T))} 
$$
Recall that $\p_x u_j(\cdot,0) \in L^2((0,T))$ for each $j=1,2$. Using the spectral expression \eqref{u_12_exp} together with the fact that $\p_x \phi_k^{(j)}(0)=1$, it follows that
 $$
 \sum_{k=1}^\infty a_k^{(1)}\left(\int_0^T\eta_m(t)\,\cos(\sqrt{\lambda_k^{(1)}}\,t)\,dt\right)= \sum_{k=1}^\infty a_k^{(2)}\left(\int_0^T\eta_m(t)\,\cos(\sqrt{\lambda_k^{(2)}}\,t)\,dt\right)
 $$
 In view of \eqref{eta_eq_lem} the above equation reduces to 
 $$
 a_m^{(1)}=0,
 $$
 which is a contradiction to \eqref{a_k_1}. This completes the proof of \eqref{claim_eigen_wave}. For the remainder of this proof we fix $\varepsilon>0$ to be small enough so that 
\begin{equation}\label{V_12_eq}
	\mathrm{supp}\,(V_1-V_2) \subset  [0,1-\varepsilon).\end{equation} 
Let $m\geq N$. In view of \eqref{claim_eigen_wave}, we note that 
 \begin{equation}\label{eigen_eq_wave}\lambda_m^{(1)}= \lambda_{b_m}^{(2)}\quad \text{for some $b_m \in \N$}.\end{equation}
% Note that the previous two equalities imply also that 
 %\begin{equation}
 %	\label{eigenfcn_eq_wave}
%	\phi_m^{(1)}(x)= \phi^{(2)}_{b_m}(x) \quad \forall\, x\in (0,\varepsilon) \quad \forall\, m\geq N.
% \end{equation}
By choosing $N_1>N$ sufficiently large, it follows from \eqref{eigen_eq_wave} together with the eigenvalue asymptotics \eqref{eigen_asymp} that there holds:
 \begin{equation}\label{eigen_eq_wave_1}\lambda_m^{(1)}= \lambda_{m}^{(2)}\quad \text{for all $m \geq N_1$}.\end{equation}
 Next, for each $m\geq N_1$, applying Lemma~\ref{lem_completeness} again, we deduce that there exists $\theta_m \in L^2((0,T))$ satisfying
 	\begin{equation}
 	\label{theta_m_eq}
 	\begin{aligned}
 		&\int_{0}^{T}\, \theta_m(t)\,\cos(\sqrt{\lambda^{(1)}_k}t)\,dt=\delta_{km} \quad &\forall\, k\in \N, \\
 		&\int_{0}^{T}\,\theta_m(t)\,\cos(\sqrt{\lambda^{(2)}_k}t)\,dt=\delta_{km} \quad &\forall\, k\in \N.
 	\end{aligned}
 \end{equation}
We deduce via \eqref{DN_eq_1} that 
$$
\left(\p_x u_1(\cdot,0),\overline{\theta_m(\cdot)}\right)_{L^2((0,T))} = \left(\p_x u_2(\cdot,0),\overline{\theta_m(\cdot)}\right)_{L^2((0,T))}. 
$$
Using the spectral expression \eqref{u_12_exp} together with the fact that $\p_x\phi_k^{(j)}(0)=1$, the previous equation implies that
$$
\sum_{k=1}^\infty a_k^{(1)}\left(\int_0^T\theta_m(t)\,\cos(\sqrt{\lambda_k^{(1)}}\,t)\,dt\right)= \sum_{k=1}^\infty a_k^{(2)}\left(\int_0^T\theta_m(t)\,\cos(\sqrt{\lambda_k^{(2)}}\,t)\,dt\right),
$$
which together with \eqref{theta_m_eq} implies that
\begin{equation}\label{a_m_12}
a_m^{(1)}= a_{m}^{(2)} \quad \forall\, m\geq N_1.
\end{equation}
Let us now consider \eqref{DN_eq_1} again and write 
$$
\left(\p_x u_1(\cdot,1),\overline{\theta_m(\cdot)}\right)_{L^2((0,T))} = \left(\p_x u_2(\cdot,1),\overline{\theta_m(\cdot)}\right)_{L^2((0,T))} 
$$
Analogously as above, the left hand side is equal to
$$
\sum_{k=1}^\infty a_k^{(1)}\left(\int_0^T\theta_m(t)\,\cos(\sqrt{\lambda_k^{(1)}}\,t)\,dt\right)\p_x\phi_k^{(1)}(1)= a_m^{(1)}\,\p_x\phi_m^{(1)}(1),
$$
while the right hand side is equal to 
$$
\sum_{k=1}^\infty a_k^{(2)}\left(\int_0^T\theta_m(t)\,\cos(\sqrt{\lambda_k^{(2)}}\,t)\,dt\right)\p_x\phi_k^{(2)}(1)= a_{m}^{(2)}\,\p_x\phi_{m}^{(2)}(1),
$$
Combining with \eqref{a_m_12}, we deduce that given each $m\geq N_1$, there holds
$$
\p_x\phi_m^{(1)}(1)= \p_x\phi_{m}^{(2)}(1)
$$
Defining the set $S= \{N_1,N_1+1,N_1+2,\ldots\}$ and with $\varepsilon$ as above, it is clear that the hypothesis of Theorem~\ref{thm_spectral} are now satisfied thus yielding that $V_1=V_2$ on $(0,1)$. The equality of the initial data $f_2=f_1$ simply follows from the unique continuation principle for wave equation and the fact that $T>2$.
\end{proof}
Let us complete the section by also providing a sketch of the proof of Theorem~\ref{thm_1_support} as it is rather similar to that of the previous theorem.
\begin{proof}[Proof of Theorem~\ref{thm_1_support}]
	The proof replicates the approach that we followed in proving Theorem~\ref{thm_1}. Indeed, as in the latter proof, we can show that for $N>1$ sufficiently large (depending only on $\|V_1\|_{L^\infty}((0,1))$ and  $\|V_2\|_{L^\infty}((0,1))$) and given any $k\geq N$ such that
	$$
	\left(f_1,\phi_k^{(1)}\right)_{L^2((0,1))}\neq 0,
	$$
	there holds:
	$$
	\lambda_k^{(1)}=\lambda_{k}^{(2)} \quad \text{and}\quad \p_x \phi_k^{(1)}(1)=\p_x\phi_{k}^{(2)}(1).
	$$
	It remains to characterize the density of the set 
	$$
	P:= \left\{ k\in \N\,:\, k\leq N \quad \text{or}\quad 	\left(f_1,\phi_k^{(1)}\right)_{L^2((0,1))}= 0\right\}.
	$$
	As $f_1$ is not identical to zero and as $\mathrm{supp}\,f_1\subset [0,\varepsilon]$, it follows from \cite[Lemma 3.4]{Feiz25} together with the asymptotic expression \eqref{eigen_asymp} for the eigenvalues that there holds:
	$$
	\limsup_{r\to \infty} \frac{|P\cap (0,r)|}{r} \leq \varepsilon.
	$$
	Thus, defining the set $S:= \N \setminus P$, it is clear that 
	$$
		\limsup_{r\to \infty} \frac{|S\cap (0,r)|}{r}\geq \liminf_{r\to \infty} \frac{|S\cap (0,r)|}{r} \geq 1-\varepsilon.
	$$
	Thus, the hypotheses of Theorem~\ref{thm_spectral} are satisfied. We conclude that $V_1=V_2$ on $(0,1)$. It follows from unique continuation that $f_1=f_2$ on $(0,1)$ as well.
	
	\end{proof}
\section{Active imaging: proof of Theorem~\ref{thm_2}}
\label{sec_active_pf}
We begin with two lemmas.
\begin{lemma}\label{lele}
	Let $\delta\in (0,1)$. Let $V\in C^4([0,1])$ and denote by $\lambda_1<\lambda_2<\ldots$, the Dirichlet eigenvalues of the operator $-\frac{d^2}{dx^2}+V(x)$ on $(0,1)$. Let $\phi_k\in C^6([0,1])$ be as defined in \eqref{eigen_functions}-\eqref{initial_data}. Let $f\in H^3_0((0,1))$, let $F\in H^1_0((0,1))$ and let $P\subset \N$ be defined by 
	\begin{equation}\label{Gamma_schrod}
		P= \left\{ k\in \N\,:\,\lambda_k =0 \quad \text{or}\quad \int_0^{\delta} \phi_k(t)\,dt=\int_0^1 (-F(t)+\lambda_k\,f(t))\phi_k(t)\,dt \right\}.		
		\end{equation}
	There holds, 
	\begin{equation}\label{P_delta}
	\limsup_{r\to \infty} \frac{|P \cap (0,r)|}{r} \leq \frac{\delta}{2}.
	\end{equation}
\end{lemma}
 
 \begin{proof}
 		Recall that there exists $\{A_\ell\}_{\ell=1}^3, \{B_\ell\}_{\ell=1}^3  \subset C^3([0,1])$ depending explicitly on $V$ such that given any $x\in [0,1]$ and $k\in \N$ there holds
 	\begin{equation*}\label{eigen_expansion_lem_schrod}
 		\phi_k(x)= \frac{\sin(k\pi x)}{k \pi}+\sum_{\ell=1}^3\left(\frac{A_\ell(x)}{k^{\ell+1}}\sin(k\pi x)+ \frac{B_\ell(x)}{k^{\ell+1}}\cos(k\pi x)\right) +O(k^{-5}),
 	\end{equation*}
 	where the modulus of convergence above (as $k\to \infty$) is uniform with respect to $x\in [0,1]$. Using \eqref{eigen_expansion_lem_schrod}, it is straightforward to see that
 	\begin{equation}\label{F_decay}\int_0^1 F(t)\,\phi_k(t) \,dt = o(\frac{1}{k^2}) \quad \text{as $k\to \infty$},\end{equation}
 	where we have used the fact that $F\in H^1_0((0,1))$ and performed integration by parts to obtain the previous bound. Analogously, it can be seen that since $f\in H^3_0((0,1))$, there holds,
 	\begin{equation}\label{f_decay}
 	\lambda_k \int_0^1 f(t)\,\phi_k(t)\,dt = o(\frac{1}{k^2}) \quad \text{as $k\to \infty$}.
 		\end{equation}
 	We also have that 
  \begin{equation}\label{phi_decay}
	\int_0^\delta \sum_{\ell=1}^3\left(\frac{A_\ell(t)}{k^{\ell+1}}\sin(k\pi t)+ \frac{B_\ell(t)}{k^{\ell+1}}\cos(k\pi t)\right)\,dt = O(\frac{1}{k^3})\quad \text{as $k\to \infty$}.
  	\end{equation}
  Combining the previous three bounds, it is clear that the claim in the lemma follows if we can prove that the set 
  $$\widetilde{P} = \left\{k\in \N\,:\, \left|\int_0^\delta \sin(k\pi t)\,dt\right|<\frac{\delta^2}{32\pi k}\right\},$$
  satisfies 
  $$	\limsup_{r\to \infty} \frac{|\widetilde{P} \cap (0,r)|}{r} \leq \frac{\delta}{2}.$$ 
  This is due to the fact that the set $P$ is a subset of the union of $\widetilde{P}$ with at most a finite number of other positive integers. Let us now study the set $\widetilde{P}$. Note that
  $$
  \int_0^\delta \sin(k\pi t)\,dt = \frac{1}{k\pi}\left(1- \cos(k\pi\delta)\right)=\frac{2}{k\pi} \sin^2(\frac{k\pi\delta}{2}).
  $$
  Therefore, $k\in \widetilde{P}$ implies that 
  $$
  |\sin(\frac{k\pi\delta}{2})|<\frac{\delta}{8},
  $$
  and thus it must be that 
  $$k\in \left( \frac{2n}{\delta}-\frac{1}{2\pi},  \frac{2n}{\delta}+\frac{1}{2\pi}\right)\quad \text{for some $n\in \N$}.$$
  The claimed inequality \eqref{P_delta} follows immediately as the above interval has at most one integer within it for each $n \in \N$.
 \end{proof}
 
\begin{lemma}
	\label{lem_completeness_schrodinger}
	For $j=1,2,$ let $V_j\in L^\infty((0,1))$ and denote by $\lambda_1^{(j)}<\lambda_2^{(j)}<\ldots$, the Dirichlet eigenvalues of the operator $-\frac{d^2}{dx^2}+V_j(x)$ on $(0,1)$. Let $m\in \N$ and suppose that 
	\begin{equation}\label{lambda_m_schrod}
		\lambda_m^{(1)}\neq 0 \quad \text{for some fixed $m\in \N$}.\end{equation} 
		Let $\delta\in (0,1)$. There exists 
	\begin{equation}\label{eta_support_schrodinger}
		\eta_m\in L^2(\mathbb R)\quad \text{with}\quad \mathrm{supp}\,\eta_m \subset [0,\delta], 
	\end{equation}
	such that
	\begin{equation}
		\begin{aligned}
			\int_0^{\delta}\, \eta_m(t) e^{\mathrm{i}\,z\,t}\,dt&=0 \quad \forall\, z\in \left(\{\lambda_k^{(1)},\lambda_k^{(2)}\}_{k=1}^\infty \right)\setminus \{\lambda_m^{(1)}\}, \\
				\int_0^{\delta}\, \eta_m(t) e^{\mathrm{i}\,\lambda_m^{(1)}\,t}\,dt&=1, \\
			\int_0^{\delta}\, \eta_m(t)\,dt &=0.
		\end{aligned}
	\end{equation}	
\end{lemma}

\begin{proof}
	Let us define for each $j=1,2,$ the set 
	$$\Gamma_j=\{\pm \lambda_k^{(j)}\}_{k\in \N} \cup \{0\}.$$
	Recalling the eigenvalue asymptotics \eqref{eigen_asymp}, we deduce that $\Gamma_j$ is uniformly discrete and that
	$$ D^+(\Gamma_j) :=\lim_{r\to \infty} \max_{x\in \R} \frac{|\Gamma_j\cap (x,x+r)|}{r}= 0<\frac{\delta}{8\pi}.$$
	It follows from combining the main result of \cite{Kahane1957SurLF} together with \cite[Theorem 7, pp. 129]{Young1980AnIT} that the set $\{e^{\pm \mathrm{i}\,\lambda^{(1)}_k\,t}\}_{k\in \N}$ is not complete in $L^4((0,\frac{\delta}{4}))$. The rest of the proof can be done analogously to the proof of Lemma~\ref{lem_completeness} but we include it here for the sake of completeness. There exists a nonzero function $\theta_m^{(j)} \in L^4(\mathbb R)$, $j=1,2,$, with supp$(\theta_m^{(j)})\subset [0,\frac{\delta}{4}]$, such that
		\begin{equation}
			\label{theta_m}
		\begin{aligned}
			\int_0^{\frac{\delta}{4}}\, \theta^{(1)}_m(t) e^{\mathrm{i}\,z\,t}\,dt&=0 \quad \forall\, z\in \{\lambda_k^{(1)}\}_{k=1}^\infty, \\
				\int_0^{\frac{\delta}{4}}\, \theta^{(2)}_m(t) e^{\mathrm{i}\,z\,t}\,dt&=0 \quad \forall\, z\in \{\lambda_k^{(2)}\}_{k=1}^\infty, \\
		\int_0^{\frac{\delta}{4}}\, \theta_m^{(j)}(t)\,dt &=0 \quad j=1,2.
		\end{aligned}
	\end{equation}	
Let us now define $\tilde{\eta}_m\in L^2(\mathbb R)$ with $\mathrm{supp}\, \tilde{\eta}_m \subset [0,\frac{\delta}{2}]$ via the convolution of the two functions above, namely
\begin{equation}\label{tilde_eta_m_schrod}
	\tilde{\eta}_m(t) = \int_{0}^{\frac{\delta}{4}} \theta^{(1)}_m(t-\tau)\,\theta^{(2)}_m(\tau)\,d\tau \quad \forall\, t\in \mathbb R.
\end{equation}
Next, we define the entire function 
\begin{equation}\label{F_def_entire_schrod}
	F_m(z) = \int_{0}^{\frac{\delta}{2}} \tilde{\eta}_m(t) e^{-\mathrm{i}zt}\,dt \quad \forall\, z\in \C.	
\end{equation}
Combining \eqref{theta_m} and the convolution definition \eqref{tilde_eta_m_schrod} it follows that 
$$
F_m(z)=0 \quad \forall\, z\in \left\{-\lambda^{(1)}_k\right\}_{k\in \N} \cup \left\{-\lambda^{(2)}_k\right\}_{k\in \N}\cup \{0\}.
$$
Suppose that $z=-\lambda_m^{(1)}$ is a zero of order $s$ for the function $F(z)$. If $F(-\lambda^{(1)}_m)$ is not zero, then we simply define $s=0$. With this in mind, let us define the entire function
$$
G_m(z) = \frac{F_m(z)}{(z+\lambda_m^{(1)})^s}.
$$
We record that 
$$
G_m(-\lambda_m^{(1)}) \neq 0.
$$
As the function $$H_m(z)=\frac{G_m(z)}{G_m(-\lambda_m^{(1)})},$$ is of exponential type $\frac{\delta}{2}$ and it is square-integrable along horizontal lines in the complex plane, it follows from the Paley-Wiener theorem that $H_m$ is the Fourier transform of some $\hat{\eta}_m \in L^2((-\frac{\delta}{2},\frac{\delta}{2}))$. Thus,
\begin{equation}
	\label{hat_eta_m_eq_schrod}
	\begin{aligned}
		\int_{-\frac{\delta}{2}}^{\frac{\delta}{2}}\, \hat{\eta}_m(t)\,e^{-\mathrm{i}zt}\,dt&=0 \quad \forall\, z\in \{-\lambda_k^{(1)},-\lambda_k^{(2)}\}_{k=1}^\infty \setminus \{-\lambda_m^{(1)}\}, \\
		\int_{-\frac{\delta}{2}}^{\frac{\delta}{2}}\, \hat{\eta}_m(t) e^{\mathrm{i}\lambda^{(1)}_mt}\,dt&=1,\\
		\int_{-\frac{\delta}{2}}^{\frac{\delta}{2}} \hat{\eta}_m(t)\,dt&=0.
	\end{aligned}
\end{equation}
 The proof is completed by defining $$\eta_m(t)=e^{-\frac{1}{2}\mathrm{i}\lambda_m^{(1)}\delta}\,\hat{\eta}_m(t-\frac{\delta}{2}) \quad  \forall\, t\in (0,\delta).$$  
	\end{proof}

\begin{proof}[Proof of Theorem~\ref{thm_2}]
	Recall from the statement of the theorem that there exists $\varepsilon \in (0,1)$ such that 
	\begin{equation}\label{V_1_2_pf}
		\mathrm{supp}\, (V_1-V_2) \subset [0,1-\varepsilon).
	\end{equation}
We will assume throughout the rest of this proof that $0<\delta<\min\{1,2\varepsilon\}$. We start by writing for each $j=1,2,$ that
	\begin{equation}\label{u_12_exp_schrod}
		u_j(t,x)=\sum_{k=1}^\infty A_k^{(j)}(t)\,\phi_k^{(j)}(x), \quad t\in (0,\delta)\quad x\in (0,1),
	\end{equation}
	where for each $k\in \N$, when $\lambda_k^{(j)}\neq0$ we have
$$	A_k^{(j)}(t) = e^{\mathrm{i}\lambda_k^{(j)}t}\left(\frac{\left(f_j, \phi_k^{(j)} \right)_{L^2((0,1))}}{\|\phi_k^{(j)}\|^2_{L^2((0,1))}}\, -\frac{1}{\lambda_k^{(j)}}\left(1-e^{-\mathrm{i}\lambda_k^{(j)}t}\right)\frac{\left(\mathbbm{1}_\delta+F_j, \phi_k^{(j)} \right)_{L^2((0,1))}}{\|\phi_k^{(j)}\|^2_{L^2((0,1))}}\right)$$
and when $\lambda_k^{(j)}=0$ we have
$$A_k^{(j)}(t)=\frac{\left(f_j, \phi_k^{(j)} \right)_{L^2((0,1))}}{\|\phi_k^{(j)}\|^2_{L^2((0,1))}}\, -\mathrm{i}t\frac{\left(F_j+\mathbbm{1}_\delta , \phi_k^{(j)} \right)_{L^2((0,1))}}{\|\phi_k^{(j)}\|^2_{L^2((0,1))}}.$$
The convergence in the above infinite series is to be understood with respect to the  
	$C([0,T];H^1_0((0,1)))$ topology. Let us define the set 
	$$
	P=  \left\{ k\in \N\,:\, \lambda_k^{(1)}=0 \quad \text{or}\quad \int_0^{\delta} \phi^{(1)}_k(t)\,dt=\int_0^1( -F_1(t)+\lambda_k^{(1)}f_1(t))\phi^{(1)}_k(t)\,dt \right\}.
	$$
	Note that in view of Lemma~\ref{lele}, there holds:
		\begin{equation}\label{P_delta_1}
		\limsup_{r\to \infty} \frac{|P \cap (0,r)|}{r} \leq \frac{\delta}{2}<\varepsilon.
	\end{equation}
We claim that 
\begin{equation}\label{claim_eigen_schrod}
	\lambda_m^{(1)} = \lambda_{b_m}^{(2)} \qquad \forall\, m\in \N\setminus P,	
	\end{equation}
for some sequence $b_1<b_2<\ldots$ of positive integers. To prove this, suppose for contrary that there exists $m\in \N \setminus P$ such that 
\begin{equation}\label{lambda_m_hypo}
	\lambda_m^{(1)} \notin \{\lambda_k^{(2)}\}_{k=1}^\infty.
	\end{equation}
Applying Lemma~\ref{lem_completeness_schrodinger} we deduce that there exists $\eta_m\in L^2((0,\delta))$ such that
	\begin{equation}\label{eta_m_schrod}
	\begin{aligned}
		\int_0^{\delta}\, \eta_m(t) e^{\mathrm{i}\,\lambda_k^{(1)}\,t}\,dt&=\delta_{km} \quad \forall\, k\in \N,\\
		\int_0^{\delta}\, \eta_m(t) e^{\mathrm{i}\,\lambda_k^{(2)}\,t}\,dt&=0, \quad \forall\, k\in \N,\\
		\int_0^{\delta}\, \eta_m(t)\,dt &=0.
	\end{aligned}
\end{equation}	
Applying \eqref{DN_eq_2} at $x=0$, and recalling that $\p_x \phi_k^{(1)}(0)=\p_x \phi_k^{(2)}(0)=1$, we deduce that
$$
\sum_{k=1}^\infty \int_0^\delta A_k^{(1)}(t)\eta_{m}(t)\,dt = \sum_{k=1}^\infty \int_0^\delta A_k^{(2)}(t)\eta_{m}(t)\,dt.
$$
Using \eqref{eta_m_schrod} it is straightforward to see that the right hand side of the above expression is zero, while the left hand side is equal to 
$$
\frac{1}{\lambda_m^{(1)}\,\|\phi_m^{(1)}\|^2_{L^2((0,1))}}\left(-\int_0^\delta \phi_m^{(1)}(t)\,dt-\int_0^1 F_1(x)\phi_m^{(1)}(t)\,dt+ \lambda_m^{(1)}\int_0^1 f_1(t)\,\phi_m^{(1)}(t)\,dt \right),
$$
which is nonzero, thanks to the fact that $m\in \N \setminus P$ together with definition of $P$. This yields a contradiction thus proving that \eqref{claim_eigen_schrod} holds. By choosing $N>1$ sufficiently large we can combine \eqref{claim_eigenfcn_schrod} with \eqref{eigen_asymp} to obtain that 
\begin{equation}\label{claim_eigen_schrod_1}
	\lambda_m^{(1)} = \lambda_{m}^{(2)} \qquad \forall\, m\in (\N\setminus P) \cap \{N,N+1,\ldots\}.	
\end{equation}
Next, we claim that 
\begin{equation}\label{claim_eigenfcn_schrod}
	\p_x\phi_m^{(1)}(1) = \p_x\phi_{m}^{(2)}(1) \qquad \forall\, m\in (\N\setminus P) \cap \{N,N+1,\ldots\}.	
\end{equation}
Let us prove \eqref{claim_eigenfcn_schrod} for a fixed $m\in (\N\setminus P) \cap \{N,N+1,\ldots\}$. By Lemma~\ref{lem_completeness_schrodinger}, there exists $\theta_m \in L^2((0,\delta))$ such that
	\begin{equation}\label{theta_m_schrod}
	\begin{aligned}
		\int_0^{\delta}\, \theta_m(t) e^{\mathrm{i}\,\lambda_k^{(1)}\,t}\,dt&=\delta_{km} \quad \forall\, k\in \N,\\
		\int_0^{\delta}\, \theta_m(t) e^{\mathrm{i}\,\lambda_k^{(2)}\,t}\,dt&=\delta_{km}, \quad \forall\, k\in \N,\\
		\int_0^{\delta}\, \theta_m(t)\,dt &=0.
	\end{aligned}
\end{equation}	
Applying \eqref{DN_eq_2} at $x=0$ again, and recalling that $\p_x \phi_k^{(1)}(0)=\p_x \phi_k^{(2)}(0)=1$, we deduce that
$$
\sum_{k=1}^\infty \int_0^\delta A_k^{(1)}(t)\theta_{m}(t)\,dt = \sum_{k=1}^\infty \int_0^\delta A_k^{(2)}(t)\theta_{m}(t)\,dt.
$$
Using \eqref{theta_m_schrod} we deduce that
\begin{equation}\label{a_1_2_eq}
a_m^{(1)}=a_m^{(2)} \neq 0,
\end{equation}
where 
$$
a_m^{(1)}=\frac{1}{\lambda_m^{(1)}\,\|\phi_m^{(1)}\|^2_{L^2((0,1))}}\left(-\int_0^\delta \phi_m^{(1)}(t)\,dt- \int_0^1 F_1(t)\phi_m^{(1)}(t)\,dt+\lambda_m^{(1)}\int_0^1 f_1(t)\,\phi_m^{(1)}(t)\,dt \right),
$$
and 
$$
a_m^{(2)}= \frac{1}{\lambda_{m}^{(2)}\,\|\phi_{m}^{(2)}\|^2_{L^2((0,1))}}\left(-\int_0^\delta \phi_{m}^{(2)}(t)\,dt-\int_0^1 F_2(t)\phi_{m}^{(2)}(t)\,dt+\lambda_{m}^{(2)}\int_0^1 f_2(t)\,\phi_{m}^{(2)}(t)\,dt \right),
$$
and we recall from the definition of the set $P$ that $a_m^{(1)} \neq 0$. Next, we apply \eqref{DN_eq_2} at $x=1$, to deduce that
$$
\sum_{k=1}^\infty \left(\int_0^\delta A_k^{(1)}(t)\theta_{m}(t)\,dt\right) \p_x \phi_k^{(1)}(1)= \sum_{k=1}^\infty\left( \int_0^\delta A_k^{(2)}(t)\theta_{m}(t)\,dt\right) \p_x \phi_k^{(2)}(1).
$$
Using \eqref{theta_m_schrod} together with \eqref{a_1_2_eq} yields the claim \eqref{claim_eigenfcn_schrod}. We may now apply Theorem~\ref{thm_spectral} with $S=\N \setminus (P\cup \{1,\ldots,N\})$ (note that \eqref{S_prop} is satisfied thanks to \eqref{P_delta_1} together with the fact that removing a finite number of elements of a set does not change its density) to conclude that $V_1=V_2$. 
\end{proof}

\appendix

\section{Spectral interpolation for Schr\"odinger operators}

Our aim in this appendix section is to provide an interpolation result for 1-D Schr\"odinger operators that may be of independent interest. Let us first give a brief outline of the classical interpolation theory related to Fourier transforms on $\mathbb R$. Let $I=[a,b]$ for some $a<b$ and define the Paley-Wiener space
\[
PW_{I} = \left\{ f \in L^2(\mathbb{R}) : f = \widehat{F} \text{ and } \mathrm{supp}\, F \subset [a,b] \right\},
\]
where $\widehat{F}$ denotes the Fourier transform of $F$. Let $\Gamma = \{\gamma_k\}_{k=1}^\infty \subset \R$ be uniformly discrete (i.e. the pairwise distance of its elements has a uniform lower bound) and that its uniform upper density satisfies
\begin{equation}\label{G_density_PW}
	D^+(\Gamma):=\lim_{r\to \infty}\max_{x\in \R}\frac{\left|\Gamma \cap (x,x+r)\right|}{r} < \frac{b-a}{2\pi}.
\end{equation}

\begin{itemize}
	\item[{\bf(Q1)}] {\em Does there exist a continuous map $L: \ell^2(\mathbb{N}) \to L^2(\mathbb{R})$ such that for any $c \in \ell^2(\mathbb{N})$, the function $f = Lc$ satisfies
	\begin{equation}\label{PW_inter}
		f \in PW_{I} \quad \text{and} \quad f(\gamma_k) = c_k \quad \text{for all } k \in \mathbb{N}?
	\end{equation}}
\end{itemize}

This question, along with related variants—such as replacing the interval $(a,b)$ with more general measurable sets or considering different function spaces—has been extensively studied; see, e.g., \cite{Beurling1962OnFT,Beurling1989TheCW,Landau1967NecessaryDC,Meyer2018MeanperiodicFA,Olevski2009InterpolationIB} for such results as well as \cite[Section 4]{Young1980AnIT} for an introduction into this topic. The precise formulation {\bf(Q1)} above was answered affirmatively by Kahane in 1957 \cite{Kahane1957SurLF}, who also showed that the density condition is nearly optimal: the inequality $D^+(\Gamma) \leq \frac{b-a}{2\pi}$ is in fact necessary. 

We aim to study an analogue of the above interpolation problem in the context of 1-D Schr\"odinger operators on a bounded interval. To be precise, we first pose an analogue of {\bf(Q1)} for Schr\"odinger operators, as follows. Let $V \in L^\infty((0,1))$ be real-valued and suppose that its Dirichlet spectrum and corresponding eigenfunctions are defined as in \eqref{eigen_functions}--\eqref{initial_data}. We now ask:

\begin{itemize}
	\item[{\bf(Q2)}] {\em Given $\varepsilon \in (0,1)$, under what assumptions on a sequence $c = (c_1, c_2, \ldots)$ and a subset $P=\{p_k\}_{k=1}^\infty \subset \mathbb{N}$ with $p_1<p_2<\ldots$, does there exist a continuous linear map $L: X \to L^2((0, \varepsilon))$, where $X$ is a suitable Banach subspace of $\ell^2(\mathbb{N})$, such that the interpolation condition
	\begin{equation}\label{interpolation_schrodinger}
		\int_0^{\varepsilon} (Lc)(x)\, \phi_{p_k}(x)\, dx = c_k \quad \text{for all } k \in \N
	\end{equation}
	is satisfied for all $c\in X$?}
\end{itemize}

To answer {\bf (Q2)}, we introduce the Hilbert space $\ell^2_P(\mathbb{N})$, defined as the completion of sequences $c = (c_1, c_2, \ldots)$ with respect to the norm
\begin{equation}\label{l21_norm}
	\|c\|_{\ell^2_P(\mathbb{N})}^2 = \sum_{k=1}^\infty p_k^2\, |c_k|^2,
\end{equation}
and endowed with the inner product
\[
(c,d)_{\ell^2_P(\mathbb{N})} = \sum_{k=1}^\infty p_k^2\, c_k \,\overline{d_k}.
\]
We remark that the interpolation question {\bf (Q2)} is equivalent to studying the surjectivity of the linear mapping 
$$
\mathcal L: L^2((0,\varepsilon)) \to \ell^2_P(\N)
$$ 
defined by 
$$
\mathcal L f = \left(\left(f, \phi_{p_1}\right)_{L^2((0,\varepsilon))},\left(f, \phi_{p_2}\right)_{L^2((0,\varepsilon))},\ldots  \right).
$$
The fact that $\mathcal Lf \in \ell^2_P(\N)$ follows from combining the asymptotic expression for eigenvalues and eigenfunctions, see e.g. \cite[Theorem 4, page 35]{Pschel1986InverseST}. %Let us also mention that the adjoint of the above mapping, denoted by $\mathcal L^\star: \ell_P^2(\N) \to L^2((0,\varepsilon))$ is given by 
%$$
%\mathcal L^\star c = \sum_{k=1}^\infty c_k \phi_{p_k}|_{(0,\varepsilon)}.
%$$
We prove the following interpolation theorem that answers {\bf (Q2)} affirmatively under certain assumption on the set $P$. Our proof also shows that the adjoint map $\mathcal L^\star$ is injective and has a closed range.

\begin{theorem}[Interpolation theorem for 1-D Schr\"odinger operators]\label{thm_interpolation}
	Let $\varepsilon \in (0,1)$ and let $P \subset \mathbb{N}$ be a countably infinite set satisfying $D^+(P) < \varepsilon$. Let $V \in L^\infty((0,1))$ be a real-valued function, and denote by $\lambda_1 < \lambda_2 < \cdots$ the Dirichlet eigenvalues of the operator $- \frac{d^2}{dx^2} + V(x)$ on $(0,1)$. Let $\{\phi_k\}_{k=1}^\infty \subset W^{2,\infty}((0,1))$ be as in \eqref{eigen_functions}--\eqref{initial_data}. Then there exists a continuous linear map $L : \ell^2_P(\mathbb{N}) \to L^2((0, \varepsilon))$ such that \eqref{interpolation_schrodinger} holds\footnote{Our proof also shows near optimality of this result; for a set $P\subset \N$ to be an interpolation set, in the sense of \eqref{interpolation_schrodinger}, it is necessary that $D^+(P)\leq \varepsilon$.}.
\end{theorem}

Note that the theorem is trivial when $P\subset \N$ is finite and just follows from linear independence of restriction of eigenfunctions $\{\phi_k\}_{k=1}^\infty$ on the set $(0,\varepsilon)$ which is why we assume that it $P$ is countably infinite. Let us also remark that if instead of Dirichlet eigenfunctions that are normalized through initial data condition \eqref{initial_data}, one works with $L^2((0,1))$-orthonormal eigenfunctions in \eqref{interpolation_eq}, then the linear map $L$ can be constructed as a continuous mapping from $\ell^2(\N)$ into $L^2((0,\varepsilon))$. Let us also mention that similar results can be obtained for  eigenfunctions of \eqref{eigen_functions} subject to Neumann or more general Robin type boundary conditions. Let us also mention that such interpolation results form a bridge between discrete data and continuous function spaces, allowing for the synthesis of missing information from partial or incomplete measurements.

\subsection{Proof of Theorem~\ref{thm_interpolation}}

In order to prove the theorem, we first need to recall an interpolation theorem due to Kahane \cite{Kahane1957SurLF} as well as a Paley-Wiener type result for 1-D Schr\"odinger operators due to Remling \cite{Remling2002SchrdingerOA,Remling2003InverseST}. 

\begin{lemma}[\cite{Kahane1957SurLF}]
	\label{lem_Kahane}
	Let $\delta>0$. Let $\Gamma=\{\gamma_k\}_{k=1}^{\infty}\subset (0,\infty)$ be a uniformly discrete set (i.e. the pairwise distances of its elements has a uniform positive lower bound) and suppose that its upper uniform density satisfies
	\begin{equation}\label{density_bound}
		D^+(\Gamma) < \frac{\delta}{\pi}.
	\end{equation}
	There exists a bounded linear operator $\mathcal I_\delta: \ell^2(\N)\to L^2((0,\delta))$ such that given any $a=(a_1,a_2,\ldots) \in \ell^2(\N)$, there holds
	\begin{equation}\label{interpolation_eq}
		\int_0^\delta \, (\mathcal I_\delta a)(x) \,\sin(\gamma_k x)\,dx=a_k \quad \forall\, k\in \N.\end{equation}
\end{lemma}

Next, we recall an spectral result due to Remling, see \cite[Theorem 2.2]{Remling2002SchrdingerOA} (cf. \cite{Remling2003InverseST}). 
\begin{lemma}
	\label{lem_Remling}
	Let $V \in L^\infty((0,1))$ and let $\tau\in \R$. Given each $z\in \C$, let $\psi(\cdot,z) \in W^{2,\infty}((0,1))$ be the unique solution to
	\begin{equation}\label{psi_expansion} -\p^2_x\psi(x,z) + V(x)\, \psi(x,z) = z \,\psi(x,z) \quad \text{on $(0,1)$},\end{equation}
	subject to $ \psi(0,z)=0$ and $\p_x \psi(0,z)=1.$ Let us define for each $\delta\in (0,1)$,
	$$ \mathscr S_\delta:= \left\{ \int_0^\delta f(x)\,\psi(x,z)\,dx\,:\, f\in L^2((0,\delta))\right\}.$$
	The set $\mathscr S_\delta$ is independent of $V$. In particular, there holds
	$$ \mathscr S_\delta=\left\{\int_{0}^\delta f(x)\,\frac{\sin(\sqrt{z+\tau}\, x)}{\sqrt{z+\tau}}\,dx\,:\, f\in L^2((0,\delta))\right\}\footnote{We are using the principal branch of the logarithm to define $\sqrt{z+\tau}$ and we are defining $\frac{\sin(\sqrt{z+\tau}\, x)}{\sqrt{z+\tau}}$ at $z=-\tau$ via continuity. Note that Remling's theorem is stated with $\tau=0$ but our version follows from considering the set $\mathscr S_N$ for a constant potential $\tau$.}.$$
\end{lemma}
In view of the above result by Remling, we define the bounded linear map 
$$ \mathcal K_{\tau,\delta} : L^2((0,\delta)) \to L^2((0,\delta))$$ 
as follows. Given any $f\in L^2((0,\delta))$, we define
\begin{equation}\label{T_map_def}
	\mathcal K_{\tau,\delta} f \in L^2((0,\delta))\end{equation} 
as the unique function that satisfies
$$
\int_{0}^\delta f(x)\,\frac{\sin(\sqrt{z+\tau} \,x)}{\sqrt{z+\tau}}\,dx=\int_0^\delta (\mathcal K_{\tau,\delta} f)(x)\,\psi(x,z)\,dx \quad \forall\, z\in \C.
$$
In particular, by setting $z$ to belong to positive real numbers and using spectral measures for 1-D Schr\"odinger operators on the unbounded interval $(0,\infty)$ (see e.g. \cite[Theorem 2.1]{Remling2002SchrdingerOA}), it follows also that there exists $C_{\tau,\delta}>0$ independent of $f$, such that
\begin{equation}\label{K_map_est}
	C_{\tau,\delta}^{-1}\, \|f\|_{L^2((0,\delta))}\leq \|\mathcal K_{\tau,\delta} f\|_{L^2((0,\delta))} \leq  C_{\tau,\delta}\, \|f\|_{L^2((0,\delta))}.
\end{equation} 

We are now ready to state the proof of our interpolation theorem.

\begin{proof}[Proof of Theorem~\ref{thm_interpolation}]
		We fix $\varepsilon\in (0,1)$, $V\in L^{\infty}((0,1))$ and also define
		\begin{equation}
			\label{def_tau}
			\tau = \|V\|_{L^\infty((0,1))}.
		\end{equation}
		In particular, 
		\begin{equation}\label{eigen_positive}
			\tau + \lambda_{k}>0 \quad \forall\, k\in \N.
		\end{equation}
			Let $c\in \ell^2_P(\N)$ be arbitrary. We will define $Lc$ as follows. First, let $\tau$ be as in \eqref{def_tau} and recall that \eqref{eigen_positive} holds. Define for each $k\in \N$, $\gamma_k=\sqrt{\lambda_{p_k}+\tau}$. By eigenvalue asymptotics, see e.g. \cite[Theorem 4, page 35]{Pschel1986InverseST}, there holds
			\begin{equation}\label{spec_asymp_0}
				\sqrt{\lambda_k+\tau}=k\pi+ O(\frac{1}{k}),  \quad \text{as $k\to \infty$},
			\end{equation}
			where the modulus of convergence to zero depends on $\|V\|_{L^\infty([0,1])}$. The asymptotic expression \eqref{spec_asymp_0} implies that the set $\Gamma=\{\gamma_k\}_{k=1}^\infty\subset (0,\infty)$ is uniformly discrete and also that 
			$$ D^+(\Gamma) <  \frac{\varepsilon}{\pi}.$$
			We apply Lemma~\ref{lem_Kahane} (with $\delta=\varepsilon$, $\Gamma$ as above and $a_k=\gamma_k\,c_k$ for all $k\in \N$) to deduce that there exists a function $\widetilde f$ (depending on $c$) satisfying
			\begin{equation}\label{interpolation_eq_pf}
				\int_0^{\varepsilon} \, \widetilde{f}(x) \,\frac{\sin(\sqrt{\lambda_{p_k}+\tau}\, x)}{\sqrt{\lambda_{p_k}+\tau}}\,dx=c_k \quad \forall\, k\in \N,\end{equation}
			and with $\|\widetilde{f}\|_{L^2((0,\varepsilon)}\leq C'\|a\|_{\ell^2(\N)}\leq C\, \|c\|_{\ell^2_P(\N)}$ for some $C>0$ independent of $c$. The proof of the theorem is now completed by defining $Lc=\mathcal K_{\tau,\varepsilon}\widetilde{f}$, thanks to Lemma~\ref{lem_Remling}. The linearity and continuity of $L$ also follows from that of the operator $\mathcal K_{\tau,\varepsilon}$.

	\end{proof}

\subsection*{Acknowledgements}
We are very grateful to Marco Marletta for helpful and stimulating discussions related to our inverse spectral result. 
\subsection*{Funding}
The work of Y. Kian is supported by the French National Research Agency ANR and Hong Kong RGC Joint Research Scheme for the project IdiAnoDiff (grant ANR-24-CE40-7039).
\subsection*{Conflict of interest}
The author has no conflicts of interest to declare that are relevant to this article.
\subsection*{Ethical statement}
This study was conducted in accordance with all relevant ethical guidelines and regulations.
\subsection*{Informed Consent}
All participants provided informed consent prior to taking part in the study.
\subsection*{Data availability statement}
Data sharing not applicable to this article as no datasets were generated or analysed during the current study.

\bibliography{refs} 

\bibliographystyle{alpha}

\end{document}